\documentclass[11pt,letterpaper]{article}

 \usepackage{amssymb,latexsym,amsmath,amsthm}

\newtheorem*{thm*}{Theorem A}

\newtheorem{thm}{Theorem}

\newtheorem{dfn}{Definition}

\newtheorem{lemma}{Lemma}

\newtheorem{remark}{Remark}

\newtheorem{prop}{Proposition}

\begin{document}

\def\d{ \partial_{x_j} }
\def\Na{{\mathbb{N}}}

\def\Z{{\mathbb{Z}}}

\def\IR{{\mathbb{R}}}

\newcommand{\E}[0]{ \varepsilon}

\newcommand{\la}[0]{ \lambda}

\newcommand{\s}[0]{ \mathcal{S}}

\newcommand{\AO}[1]{\| #1 \| }

\newcommand{\BO}[2]{ \left( #1 , #2 \right) }

\newcommand{\CO}[2]{ \left\langle #1 , #2 \right\rangle}

\newcommand{\R}[0]{ \IR\cup \{\infty \} }

\newcommand{\co}[1]{ #1^{\prime}}

\newcommand{\p}[0]{ p^{\prime}}

\newcommand{\m}[1]{   \mathcal{ #1 }}

\newcommand{ \W}[0]{ \mathcal{W}}

%  norm of H
\newcommand{ \A}[1]{ \left\| #1 \right\|_H }

% inner product H
\newcommand{\B}[2]{ \left( #1 , #2 \right)_H }

% H^* , H pairing
\newcommand{\C}[2]{ \left\langle #1 , #2 \right\rangle_{  H^* , H } }

 \newcommand{\HON}[1]{ \| #1 \|_{ H^1} }

% Omega    \Om
\newcommand{ \Om }{ \Omega}

% \partial Omega      \pOm
\newcommand{ \pOm}{\partial \Omega}

%   D(\Omega)   \D
\newcommand{\D}{ \mathcal{D} \left( \Omega \right)}

% D'( Omega)        \DP
\newcommand{\DP}{ \mathcal{D}^{\prime} \left( \Omega \right)  }

% D' pairing
\newcommand{\DPP}[2]{   \left\langle #1 , #2 \right\rangle_{  \mathcal{D}^{\prime}, \mathcal{D} }}

% (H^1)^* , H^1    (( pairing ))      \PHH
\newcommand{\PHH}[2]{    \left\langle #1 , #2 \right\rangle_{    \left(H^1 \right)^*  ,  H^1   }    }

%  H^{-1} , H_0^1  (( pairing ))   \PHO
\newcommand{\PHO}[2]{  \left\langle #1 , #2 \right\rangle_{  H^{-1}  , H_0^1  }}

 %  H^1(\Omega)     \HO
 \newcommand{\HO}{ H^1 \left( \Omega \right)}

%  H_0^1( \Omega)       \HOO
\newcommand{\HOO}{ H_0^1 \left( \Omega \right) }

% C_c^\infty(omega)
\newcommand{\CC}{C_c^\infty\left(\Omega \right) }

%H_0^1(Omega)  norm
\newcommand{\N}[1]{ \left\| #1\right\|_{ H_0^1  }  }

%H_0^1(Omega)   innerproduct
\newcommand{\IN}[2]{ \left(#1,#2\right)_{  H_0^1} }

% H^1(\Omega) inner product
\newcommand{\INI}[2]{ \left( #1 ,#2 \right)_ { H^1}}

% (H^1(\Omega))^*
\newcommand{\HH}{   H^1 \left( \Omega \right)^* }

% ( H^{-1}(\Omega))
\newcommand{\HL}{ H^{-1} \left( \Omega \right) }

\newcommand{\HS}[1]{ \| #1 \|_{H^*}}

\newcommand{\HSI}[2]{ \left( #1 , #2 \right)_{ H^*}}

\newcommand{\WO}{ W_0^{1,p}}
\newcommand{\w}[1]{ \| #1 \|_{W_0^{1,p}}}

\newcommand{\ww}{(W_0^{1,p})^*}

\newcommand{\Ov}{ \overline{\Omega}}

\title{Stability of entire solutions to supercritical elliptic problems involving advection}
\author{Craig Cowan\\
{\it\small Department of Mathematical Sciences}\\
{\it\small University of Alabama in Huntsville}\\
{\it\small 258A Shelby Center}\\
\it\small Huntsville, AL 35899 \\
{\it\small ctc0013@uah.edu} }

\maketitle

\vspace{3mm}

\begin{abstract} We examine the equation given by
\begin{equation} \label{eq_abstract}
-\Delta u + a(x) \cdot \nabla u = u^p \qquad \mbox{in $ \IR^N$,}
\end{equation}  where $p>1$ and  $ a(x)$ is a smooth  vector field satisfying some decay conditions.  We show that for $ p < p_c$, the Joseph-Lundgren exponent, that there is no positive stable solution of  (\ref{eq_abstract}) provided one imposes a smallness condition on $a$ along with a divergence free condition. In the other direction we show that for $ N \ge 4$ and $ p > \frac{N-1}{N-3}$ there exists a positive solution of (\ref{eq_abstract}) provided $a$ satisfies a smallness condition.  For $ p>p_c$ we show the existence of a positive stable solution of (\ref{eq_abstract}) provided $a$ is divergence free and satisfies a smallness condition.
 \end{abstract}

\noindent
{\it \footnotesize 2010 Mathematics Subject Classification}. {\scriptsize }

\section{Introduction and  results}

In this article we are interested the existence versus nonexistence of positive stable solutions of
\begin{equation} \label{eq}
-\Delta u + a(x) \cdot \nabla u = u^p \qquad \mbox{in $ \IR^N$,}
\end{equation}  where $p>1$ and  $ a(x)$ is a smooth  vector field satisfying some decay conditions.  We now define the notion of stability and for this we prefer to work on a general domain.

\begin{dfn} Let $ u $ denote a nonnegative smooth solution of (\ref{eq}) in an open set $\Omega \subset \IR^N$.   We say $ u$ is a stable solution of (\ref{eq}) in $ \Omega$  provided there is some smooth positive function $ E$ such that
\begin{equation} \label{linearized}
-\Delta E + a(x) \cdot \nabla E   \ge p u^{p-1} E  \qquad \mbox{in $ \Omega$.}
\end{equation}
\end{dfn}

 We begin by recalling some facts in the case where $a(x)=0$.
There has been much work done on the existence and nonexistence of positive classical solutions of
\begin{equation} \label{lane_class}
-\Delta u = u^p, \qquad \mbox{in $ \IR^N$.}
\end{equation}  For $ N \ge 3$ there exists a critical value of $p$, given by $ p_S= \frac{N+2}{N-2}$, such that for $ 1 <p < p_S$ there is no positive classical solution of (\ref{lane_class}) and for $  p > p_S$ there exist positive classical solutions,  see \cite{Caf,chen,gidas,Gidas}.
 By definition we call a nonnegative solution $u$ of (\ref{lane_class}) stable if
\begin{equation} \label{stable} \int p u^{p-1} \phi^2 \le \int |\nabla \phi|^2 \qquad \forall \phi \in C_c^\infty(\IR^N),
\end{equation} which is nothing more than the stability of $u$ using (\ref{linearized}), after using a variational principle.
The additional requirement that the solution be stable drastically alters  the existence versus nonexistence results.     It is known that there is a new critical exponent, the so called Joseph-Lundgren exponent $p_{c}$,  such that for all $ 1 <p < p_{c}$ there is no positive stable solution of (\ref{lane_class}) and for $ p>p_{c} $ there exists positive stable solutions of (\ref{lane_class}).    The value of the $p_c$ is given by
\begin{equation*}
p_c= \left\{
\begin{array}{lr}
\frac{   (N-2)^2-4N+8\sqrt{N-1}}{(N-2)(N-10)} & \qquad N \ge 11 \\
\infty & \qquad  3 \le N \le 10.
\end{array}
\right.
\end{equation*}   The first implicit appearance of $p_c$ was in the work \cite{Joseph_lundgren} where they examined $ -\Delta u = \lambda (u+1)^p$ on the unit ball in $ \IR^N$ with zero Dirichlet boundary conditions.   The exponent  $p_c$ first explicitly appeared in the works
  \cite{Wang_solo, Gui_Ni_Wang}  where they examined  the stability of radial solutions to a parabolic version of (\ref{lane_class}).  Their results easily imply the existence of a positive radial stable solution of (\ref{lane_class})  when $ p >p_c$ and the nonexistence of positive radial stable solutions in the case of $  p < p_c$.
More recently there has been interest in finite Morse index solutions of either (\ref{lane_class}) and the generalized version given by
\begin{equation} \label{far}
-\Delta u = |u|^{p-1} u, \qquad \mbox{in $ \IR^N$.}
\end{equation} In \cite{farina} they completely classified the finite Morse index solutions of (\ref{far}) and  again the critical exponent $p_c$ was involved.
For results regarding singular nonlinearities, general nonlinearities, or quasilinear equation see \cite{ces,e1,e2,egg,zz,cabreent}.

In the  work \cite{Cowan_fazly}  the nonexistence of nontrivial solutions of
\[ -div( \omega_1 \nabla u) = \omega_2 u^p \qquad \mbox{ in $ \IR^N$},\] was examined where  $ \omega_i$ are some nonnegative functions. In the special case where $ \omega_1=\omega_2$ this equation reduces to
 \begin{equation} \label{var}
 -\Delta u + \nabla \gamma(x) \cdot \nabla u = u^p \qquad \mbox{in $ \IR^N$},
  \end{equation} where $ \gamma$ is a scalar function.   Even though (\ref{var}) and (\ref{eq}) are similar a major difference is that (\ref{var}) is variational in nature; critical points of
  \[ E(u)= \frac{1}{2} \int e^{-\gamma} |\nabla u|^2 - \frac{1}{p+1} \int e^{-\gamma} |u|^{p+1},\] are solutions of (\ref{var}). This variational structure of (\ref{var})  allows one to prove various nonexistence results for (\ref{var}) by slightly modifying the nonexistence proofs used in proving similar results for  $-\Delta u = u^p $ in $ \IR^N$.  This approach will generally not work for (\ref{eq}) since in general there will not be a variational structure.

In \cite{advection} the regularity of the extremal solution, $u^*$, associated with problems of the form
 \begin{eqnarray*}
 \left\{ \begin{array}{lcl}
\hfill   -\Delta u +a(x) \cdot \nabla u   &=& \lambda f(u) \qquad \mbox{ in } \Omega  \\
\hfill u &=& 0 \qquad  \qquad \quad \mbox{ on } \pOm,
\end{array}\right.
  \end{eqnarray*}  was examined for various nonlinearities $f$.   Here $a(x)$ was an arbitrary smooth advection and the main difficulty  was to to utilize the stability of $u^*$ in a meaningful way.  As mentioned earlier, this is not a problem when $a(x)$ is the gradient of a scalar function.  The main tool used was the generalized Hardy inequality from \cite{craig}.   This same approach was extended to more general nonlinearities in \cite{advect_2}.

We now list our results.
\begin{thm} \label{smallness}   Suppose $ 3 \le N \le 10$ or $ N \ge 11$ and $ 1 < p < p_c$. Suppose $a(x)$ is a smooth divergence free vector field satisfying $ | a(x) |  \le \frac{C}{|x|+1}$ with $0<C$ sufficiently small. Then there is no positive stable solution of (\ref{eq}).

\end{thm}  The next result gives a decay estimate in the case of $ p < p_c$.  We are including this result since it may allow one to use a Lane-Emden type of change of variables to obtain a nonexistence  result without a smallness condition on the advection.

\begin{thm} \label{decay}   Suppose $ \frac{N+2}{N-2} <p <p_c$, $a(x)$ is a smooth divergence free vector field with $ |a(x)| \le \frac{C}{|x|+1}$ and $ |a| \in L^N(\IR^N)$.  Then any positive stable solution $u$ of (\ref{eq}) satisfies
\begin{equation} \label{atinf}
 \lim_{|x| \rightarrow \infty} |x|^\frac{2}{p-1} u(x)=0.
 \end{equation}
\end{thm}  The approach to solve Theorem \ref{smallness} will be to combine the methods used in \cite{farina} with the techniques from \cite{advection} which relied on generalized Hardy inequalities from \cite{craig}.    The same approach will be used in the proof of Theorem \ref{decay} with an added scaling argument.

Our final result gives an existence result.
\begin{thm} \label{existence} \begin{enumerate} \item Suppose $ N \ge 4$, $ p > \frac{N+1}{N-3}$ and $ a(x)$ is some  smooth  vector field with $ |a(x)|  \le \frac{C}{|x|+1}$.  If $ 0<C$ is sufficiently small there exists a positive solution  of (\ref{eq}).

\item  Suppose $ N \ge 11$, $ p > p_c$ and let $ a(x)$ denote some  smooth divergence free vector field with $ |a(x)|  \le \frac{C}{|x|+1}$.  For $ 0 <C$ sufficiently small (\ref{eq}) has a positive stable solution.

\end{enumerate}
\end{thm}

 The idea of the proof will be to look for a solution $u$ as a perturbation of the positive radial solution $w$ of $-\Delta w=w^p$ in $ \IR^N$ with $ w(0)=1$.  See the beginning of Section \ref{existsec} for details on $w$.
    The framework we will use to prove the existence of a positive solution will be the approach developed in \cite{davila}.    Their interest was in the existence of positive solutions of $ -\Delta u = u^p$ in $ \Omega \subset \IR^N $ an exterior domain with zero Dirichlet boundary conditions.  \\

  \noindent
  \textbf{Open Problem.}  It would be interesting to see if these smallness conditions on $a(x)$ can be removed, possibly at the expense of  adding some additional decay requirements.

\section{Nonexistence proofs}

\begin{remark}  A computation shows that $ p <p_c$ is equivalent to the condition
\begin{equation} \label{cond}
 \frac{N}{2} < 1 + \frac{2p}{p-1} + \frac{2}{p-1} \sqrt{p^2-p}.
\end{equation}  For our nonexistence results it will be easier to deal with (\ref{cond}).
\end{remark}

Theorem \ref{smallness} and Theorem \ref{decay} will depend on the following energy estimate,  which we state for a general domain.

\begin{prop} \label{prop_1} Suppose $u$ is a smooth  positive  stable  solution of (\ref{eq}) and $a(x)$ is smooth divergence free vector field.  Then for all $ 1 \le T$, $0 < \beta <1$, $0<\E$, $0<\delta$, $ \frac{1}{2}<t$ and $ 0 \le \psi \in C_c^\infty(\Omega)$ we have

\begin{eqnarray} \label{first}  \left( \beta p - \frac{Tt^2}{2t-1} \right) \int u^{2t+p-1} \psi^2 &+& \beta (1-\beta-\E) \int \frac{ | \nabla E|^2}{E^2} u^{2t} \psi^2 \nonumber \\
&&+ (T-1) \int | \nabla (u^t \psi)|^2 \nonumber \\
& \le & \left( \frac{\beta}{4 \E} + \frac{T t \delta}{2t-1} \right) \int | a|^2 u^{2t} \psi^2  \nonumber \\
&&+ \left( T + \frac{Tt}{4 \delta(2t-1)} \right) \int u^{2t} | \nabla \psi|^2 \nonumber \\
&&+ \frac{T |t-1|}{2(2t-1)} \int u^{2t} | \Delta \psi^2|.
\end{eqnarray}
\end{prop}

Define the following parameters
\[ t_-(p)= p - \sqrt{p^2-p} \quad \mbox{and} \quad  \quad t_+(p)= p+ \sqrt{p^2-p}.\] A computation shows that for $ t_-(p) <t<t_+(p)$ we have $ p- \frac{t^2}{2t-1}>0$.  This restriction on $t$ will be related to the restrictions on $t$ we must impose if one wants to obtain an estimate from Proposition \ref{prop_1}.

\noindent
\textbf{Proof of Proposition \ref{prop_1}.}  Suppose $u$ is a smooth positive  stable solution  of (\ref{eq}) in $\Omega$ and let $E>0$ satisfy (\ref{linearized}).  From   \cite{craig} we have the following generalized Hardy inequality
 \begin{equation} \label{hardy_mine}
 \beta \int \frac{-\Delta E}{E} \phi^2 + (\beta - \beta^2) \int \frac{| \nabla E|^2}{E^2} \phi^2 \le \int | \nabla \phi|^2, \qquad \forall \phi \in C_c^\infty(\Omega),
     \end{equation} for all $  \beta \IR$.   Adding  $  T \int | \nabla \phi|^2$ to both sides of the inequality,  using the fact that $E$ satisfies (\ref{linearized}) and taking $\phi= u^t \psi$ where $ \psi \in C_c^\infty(\Omega)$ gives
\begin{eqnarray*}
 \beta p \int u^{p-1} u^{2t+p-1} \psi^2 - \beta \int \frac{a \cdot \nabla E}{E} u^{2t} \psi^2 && \\
 + (\beta - \beta^2) \int \frac{| \nabla E|^2}{E^2} u^{2t} \psi^2 + (T-1) \int | \nabla (u^t \psi)|^2 && \\
 && \le T \int | \nabla (u^t \psi)|^2.
 \end{eqnarray*}  Note that the right side expands as 
 \[ T t^2 \int u^{2t-2} | \nabla u|^2 \psi^2 + 2tT \int u^{2t-1} \psi \nabla u \cdot \nabla \psi + T \int u^{2t} | \nabla \psi|^2.\]  We now wish to eliminate the term $ \int u^{2t-2} | \nabla u|^2 \psi^2$ from the inequality.  To do this we  multiply (\ref{eq}) by $ u^{2t-1} \psi^2$ and integrate over $ \Omega$ to arrive at
 \begin{eqnarray*}
  (2t-1) \int u^{2t-2} | \nabla u|^2 \psi^2 &=& \int u^{p+2t-1} \psi^2 - \int a\cdot \nabla u u^{2t-1} \psi^2 \\
  && - 2 \int \nabla u \cdot \nabla \psi u^{2t-1} \psi.
  \end{eqnarray*}  Using this equality we replace the desired term in the inequality to arrive at an inequality of the form
 \begin{eqnarray} \label{arr}
 \left( \beta p - \frac{T t^2}{2t-1} \right) \int u^{2t+p-1} \psi^2 + \beta(1-\beta) \int \frac{ | \nabla E|^2}{E^2} u^{2t} \psi^2 && \nonumber \\
 +(T-1) \int | \nabla (u^t \psi)|^2 & \le & T \int u^{2t} | \nabla \psi|^2 \nonumber  \\
 && + \sum_{k=1}^3 I_k
 \end{eqnarray}  where
 \[ I_1 = \left( 2Tt - \frac{2T t^2}{2t-1} \right) \int u^{2t-1} \psi \nabla u \cdot \nabla \psi,\]
 \[ I_2 = - \frac{T t^2}{2t-1} \int a(x) \cdot \nabla u u^{2t-1} \psi^2,\]
 \[ I_3 = \beta \int \frac{ a(x) \cdot \nabla E}{E} u^{2t} \psi^2.\]
 An integration by parts shows that
 \[ I_1= \frac{T(1-t)}{2(2t-1)} \int u^{2t} \Delta (\psi^2).\] An integration by parts shows that
 \[ |I_2| \le \frac{Tt}{2t-1} \int |a| \psi | \nabla \psi| u^{2t},\] and an application of Young's inequality shows this is less than or equal
 \[ \frac{Tt \delta}{2t-1} \int |a|^2 \psi^2 u^{2t} + \frac{Tt}{(2t-1) 4 \delta} | \nabla \psi|^2 u^{2t}.\]  An application of Young's inequality shows that
 \[ |I_3| \le \beta \E \int \frac{ | \nabla E|^2}{E^2} u^{2t} \psi^2 + \frac{\beta}{4 \E} \int |a|^2 u^{2t} \psi^2.\]   Using these upper bounds in (\ref{arr}) and regrouping gives the desired result.

\hfill $\Box$

\noindent
\textbf{Proof of Theorem \ref{smallness}.} We assume that $u$ is a positive stable solution of (\ref{eq}). Firstly note that
\[ \int |a|^2 u^{2t} \psi^2 \le C^2 \int \frac{u^{2t} \psi^2}{|x|^2},\] after considering the conditions on $a$.  Also note by Hardy's inequality we have
\[ \int | \nabla (u^t \psi)|^2 \ge C_N \int \frac{u^{2t} \psi^2}{|x|^2},\]  where $ C_N= \frac{(N-2)^2}{4}$.   Putting these into (\ref{first}) gives
\begin{eqnarray}
 \label{second}  \left( \beta p - \frac{Tt^2}{2t-1} \right) \int u^{2t+p-1} \psi^2 && \nonumber  \\
 +  \beta (1-\beta-\E) \int \frac{ | \nabla E|^2}{E^2} u^{2t} \psi^2 \nonumber  \\
+ C_1 \int \frac{u^{2t} \psi^2}{|x|^2}
& \le & C_2 \int u^{2t} \left( | \nabla \psi|^2 + | \Delta (\psi^2)| \right)
\end{eqnarray} where
\[ C_1= (T-1)C_N - C^2 \left( \frac{\beta}{4 \E} + \frac{T t \delta}{2t-1} \right),\] and $ C_2= C_2(T,t,\delta)$.   Note that for each  $  t_-(p) <t< t_+(p)$ we have $  \beta p - \frac{Tt^2}{2t-1}>0$ provided $ \beta <1$ and $ T>1$ are chosen sufficiently close to $1$.  We now pick $ \E>0$ small enough such that $ 1-\beta -\E >0$.  We now assume $C>0$ is sufficiently small such that $ C_1 \ge 0$.  We then arrive at an estimate of the form
 \begin{equation} \label{ten}
 \left( \beta p - \frac{Tt^2}{2t-1} \right) \int u^{2t+p-1} \psi^2 \le C_2 \int u^{2t} \left( | \nabla \psi|^2 + | \Delta (\psi^2)| \right),
 \end{equation} for all $ \psi \in C_c^\infty(\IR^N)$.  We now assume that $ \phi$ is a smooth cut-off function with, $ 0 \le \phi \le 1$,   $ \phi=1$ in $ B_R$ and compactly supported in $ B_{2R}$ such that $ | \nabla \phi| \le \frac{C}{R}$ and $ |\Delta \phi| \le \frac{C}{R^2}$ where $ C$ is  independent    of $R$.  Putting $\psi=\phi^m$  where $ m$ is a large integer into (\ref{ten}) gives
 \[ \left( \beta p - \frac{Tt^2}{2t-1} \right) \int u^{2t+p-1} \phi^{2m} \le C_2 C_m \int u^{2t} \phi^{2m-2} \left( | \nabla \phi|^2 + |\Delta \phi| \right),\] where $C_m$ depends only on $m$.  We now apply H\"older's inequality to see the right hand side of this inequality is bounded above by
 \[  C_2 C_m \left( \int u^{2t+p-1} \phi^\frac{(m-1) (2t+p-1)}{t} dx \right)^\frac{2t}{2t+p-1} \left( \int ( | \nabla \phi|^2 + | \Delta \phi|)^\frac{2t+p-1}{p-1} dx \right)^\frac{p-1}{2t+p-1}.\]  Now note that for sufficiently large $m$ we have that $\frac{(m-1) (2t+p-1)}{t} > 2m$ and hence we can replace the first term on the right hand side of the inequality with 
 \[\left( \int u^{2t+p-1} \phi^{2m} dx \right)^\frac{2t}{2t+p-1},\] which allows one to cancel terms to arrive at 
 
 \[ \left( \beta p - \frac{Tt^2}{2t-1} \right)^\frac{2t+p-1}{p-1} \int  u^{2t+p-1} \phi^{2m}  \le \tilde{C}_m \int \left( | \nabla \phi|^2 + | \Delta \phi| \right)^\frac{2t+p-1}{p-1}.\] We now take into account the support of $ \phi$ and how $ \phi$ scales to arrive at
 \[ \int_{B_R} u^{2t+p-1} \le C_0 R^{ N-2- \frac{2(2t+p-1)}{p-1}},\] where $C_0$ depends on the various parameters but is independent of $ R$.   Now provided $N-2- \frac{2(2t+p-1)}{p-1}<0$ we can send $ R \rightarrow \infty$ to arrive at a contradiction.   Now note we can pick a $t \in (t_-(p),t_+(p))$ such that this exponent is negative provided
 \[ \frac{N(p-1)}{2} < 2 \left( p + \sqrt{p^2-p} \right)+p-1,\] which is precisely (\ref{cond}).

\hfill $ \Box$

\textbf{Proof of Theorem \ref{decay}.} Suppose $ 0 < u$ is a smooth stable solution of (\ref{eq})  and $E>0$ solves (\ref{linearized}).
  Let $ |x_k | \rightarrow \infty$ and set $ r_k:= \frac{|x_k|}{4}$.  By passing to a subsequence we can assume that $ \{ B(x_k,r_k):  k \ge 1 \}$ is a disjoint family of balls.  We now define the rescaled functions
  \[ u_k(x)= r_k^\frac{2}{p-1} u(x_k + r_k x), \qquad a_k(x)=r_k a(x_k + r_k x), \quad E_k(x)= E(x_k+r_k x),\]  and we restrict $|x|<2$.   Then  equation (\ref{eq}) and (\ref{linearized}) are satisfied on  $B_2$ with $ u_k,a_k,E_k$ replacing $u,a,E$.
  Note that $ a_k(x)$ is a sequence of smooth divergence free vector fields which satisfy the  bound $ |a_k(x)| \le C$ for all $ |x| <2$.   From this we see the term involving $a_k$ in (\ref{first}) will be a lower order term as far as powers of $u$ are concerned and hence will cause no issues. 
    With the conditions on $N$ and $p$ there is some $  t_-(p)<t<t_+(p)$  such that  $2t+p-1 > \frac{N}{2}(p-1)>0$ and by taking $T=1$ (we can take $T=1$ since the advection term is lower order) and $ \beta<1$  sufficiently close to $1$ we can assume $\beta p - \frac{ t^2}{2t-1}>0$.
 Let $ 0 \le \phi \le 1$ be compactly supported in $B_2$ with $ \phi=1$ on $B_1$ and put $ \psi=\phi^m$, where $m$ a large integer, into (\ref{first}) where now $u,a,E$ are given by $ u_k,a_k,E_k$.   Arguing as in the proof of Theorem \ref{smallness} one can obtain a bound of the form
 \[ \int_{B_1} u_k^{2t+p-1}  \le C_0 ,\] where $ C_0$ depends on the various parameters but is independent of $k$.  
    Now note that $u_k>0$ is a sequence of smooth positive solutions of
  \[ -\Delta u_k + a_k (x) \cdot \nabla u_k = C_k(x) u_k \qquad \mbox{ in } B_2,\] where  $ C_k(x)=u_k^{p-1}$.   The above integral estimate shows that $C_k$ is bounded in $L^q(B_1)$ for some $ q> \frac{N}{2}$.
    We can now apply a  Harnack inequality from \cite{harnack} to see that
   \begin{equation} \label{second}
   \sup_{B_\frac{1}{2}} u_k \le C \inf_{B_{\frac{1}{2}}} u_k.
   \end{equation}
   If we can show that  $\inf_{B_\frac{1}{2}} u_k \rightarrow 0$ then one has  $ \sup_{B_\frac{1}{2}} \rightarrow 0$ and in particular this gives
   \[ |x_k|^\frac{2}{p-1} u(x_k) \le 4^\frac{2}{p-1} \sup_{B_\frac{1}{2}} u_k \rightarrow 0\] which gives us the desired decay estimate.     To show  $\inf_{B_\frac{1}{2}} u_k \rightarrow 0$ we will show  \[ \int_{B_1} u_k^\frac{(p-1)N}{2}  \rightarrow 0.\]
    Using a change of variables shows that
   \[ \int_{B_1} u_k^\frac{(p-1)N}{2} = \int_{B(x_k,r_k)} u^\frac{(p-1)N}{2},\] and if we show that $ u \in L^\frac{(p-1)N}{2}(\IR^N)$ then we'd have the desired result since
   \[ \int_{\IR^N} u^\frac{(p-1)N}{2} \ge \sum_{k=1}^\infty \int_{B(x_k,r_k)} u^\frac{(p-1)N}{2}. \]
Towards this we now set $ t= \frac{(p-1)(N-2)}{4}$ and note that the condition on $N$ and $p$  imply that  $ t_-(p)<t<t_+(p)$.  We now pick $ \beta <1$ but sufficiently close such that $ \beta p - \frac{t^2}{2t-1}>0$ and pick $ \E>0$ sufficiently small such that $ 1-\beta -\E>0$.   Let $ \phi$ be the smooth cut-off function from the proof of Theorem \ref{smallness},  which is equal to $1$ in $B_R$ and compactly supported in $B_{2R}$.
  We now put $ \psi=\phi^m$, where $m$ is a large integer, into (\ref{first})  taking $T=1$,  to arrive at inequality of the form
   \begin{equation} \label{hundred}
   \int u^{2t+p-1} \phi^{2m} \le C_0 \int |a|^2 u^{2t} \phi^{2m} + C_0 \int u^{2t} \phi^{2m-2} \left( | \nabla \phi|^2+ |\Delta \phi| \right).
     \end{equation}
     We now let $ \tau$ be such that $ 2 t \tau = 2t+p-1$ and let $ \tau'$ denote the conjugate index of $ \tau$.  Applying H\"older's inequality to the right hand side of (\ref{hundred}) and arguing as in the proof of Theorem \ref{smallness} we arrive at an inequality, for sufficiently large $m$,   of the form
     \[ \int u^{2t+p-1} \phi^{2m} \le C_0 \int_{B_{2R}} |a|^{2 \tau'} + C_0 \int \left( | \nabla \phi|^2 + | \Delta \phi| \right)^{\tau'},\] where $C_0$ is a constant which depends on the various parameters but is independent of $ R$.   A computation shows that $ \tau'= \frac{N}{2}$ and $ 2t+p-1= \frac{N}{2}(p-1)$.  Using these explicit values and the scaling of $ \phi$ we arrive at
     \[ \int_{B_R} u^{ \frac{N(p-1)}{2}} \le C_0 \int_{B_{2R}} |a|^N + C_0,\] and from this we obtain the desired bound on $u$ after recalling that $ |a| \in L^N(\IR^N)$.

\hfill $ \Box$

\section{Existence proofs} \label{existsec}

\noindent
 \textbf{The positive radial solution.}\\ For $ p > \frac{N+2}{N-2}$ let  $w=w(r)$ denote
  the positive radial decreasing solution of $ -\Delta w = w^p$ in $ \IR^N$  with $ w(0)=1$.  Asymptotics of $ w$ as $ r \rightarrow \infty$ are given by
  \[  w(r)= \beta^\frac{1}{p-1} r^\frac{-2}{p-1}(1+o(1)),\]  where
   \[ \beta=\beta(p,N)= \frac{2}{p-1} \left( N-2- \frac{2}{p-1} \right).\]
   In the case where $ p>p_c$ the refined asymptotics are given by
  \[ w(r)= \beta^\frac{1}{p-1} r^\frac{-2}{p-1} + \frac{a_1}{r^{\mu_0^-}} + o \left( \frac{1}{r^{\mu_0^-}}\right),\]
  where $ a_1 <0$ and $ \mu_0^- > \frac{2}{p-1}$; see \cite{Gui_Ni_Wang}.   \\

We begin by analysing the radial solution $w$ as defined above. Let $ v(r)=\beta^\frac{1}{p-1} r^\frac{-2}{p-1}$ where $ \beta$ is defined as above.

  \begin{lemma} \label{compar} Suppose $ p >p_c$, $v(r)=\beta^\frac{1}{p-1} r^\frac{-2}{p-1}$ and $ \beta$ is defined as in the definition of $w$.
   \begin{enumerate} \item  Then $ v \ge w$ in $ \IR^N$.
   \item There is some $ \E>0$ such that
 \begin{equation} \label{superstable}
 \int (p+\E) w^{p-1} \phi^2 \le \int | \nabla  \phi|^2 \qquad \forall \phi \in C_c^\infty(\IR^N).
 \end{equation}

   \end{enumerate}

  \end{lemma}

  \begin{proof} 1) Note that $ v(r) >w(r)$ for large $ r$ and small $ r$.  Towards a contradiction we assume that there is $ 0<r_0 < r_1$ such that $ w(r) > v(r) $ for all $ r_0 <r<r_1$ with $ w=v$ at $ r=r_0,r_1$.   A computation shows that for $ p >p_c$ there is some $ \E>0$ such that  $ (p+\E)\beta \le \frac{(N-2)^2}{4}$ and then from Hardy's inequality we obtain
  \begin{equation} \label{thi}
   \int (p+\E)  v^{p-1} \phi^2 \le \int | \nabla \phi|^2 \qquad \forall \phi \in C_c^\infty(\IR^N).
   \end{equation}   From this we see that
  $ v$ is a stable singular solution of $ -\Delta v = v^p$ in $ \IR^N$ and in particular its a stable solution of 
  \[ -\Delta v = v^p \mbox{ in }  r_0 <r<r_1 \qquad \mbox{ with $ v=w$ on $ r=r_0,r_1$.}\]  It is possible to use the stability of $v$ to show that $v$ is the minimal solution of this equation with the given prescribed boundary conditions.   This fact relies on the strict convexity of the nonlinearity.  Noting that $w$ satisfies the same equation with the prescribed boundary conditions one must have $ v \le w$ on $ r_0 <r<r_1$ since $v$ is a minimal solution.  This gives us the desired contradiction.  \\
  2) The result is immediate after  combining the pointwise comparison between $w$ and $v$ and using (\ref{thi}).
  \end{proof}
For the remainder $ w$ always refers to the above radial solution and $L$ to the linear operator    $ L(\phi)=-\Delta \phi - p w^{p-1} \phi$.

We now define the various   function spaces.
 For $ \sigma >0$ but small, define
\[ \| \phi\|_{\tilde{X}_\sigma} := \sup_{|x| \le 1} |x|^\sigma |\phi(x)| + \sup_{|x| \ge 1} |x|^\frac{2}{p-1} | \phi(x)|,\]  and
\[ \| f\|_{Y_\sigma} := \sup_{|x| \le 1} |x|^{\sigma+2} |f(x)| + \sup_{|x| \ge 1} |x|^{\frac{2}{p-1}+2} |f(x)|.\]    Let $ \tilde{X}_\sigma$ and $ Y_\sigma$ denote the completions of $ C_c^\infty(\IR^N \backslash \{0\})$ under the appropriate norms.   \\

 The following linear estimate is from \cite{davila} and is a key starting point for their work. They also obtain results in the case of $ \frac{N+2}{N-2} <p < \frac{N+1}{N-3}$ in \cite{davila} and also in another of their works \cite{davila_fast}.   This case is harder to deal with but luckily  are main interest is in the case of $ p >p_c$ which allows us to avoid the harder case.

\begin{thm*}  \cite{davila} Suppose $ N \ge 4$ and  $ p> \frac{N+1}{N-3}$.  There exists some small $ \sigma >0$ such that for any $ f \in Y_\sigma$ there exists some $ \phi \in \tilde{X}_\sigma $ such that $L(\phi)=f$ in $ \IR^N$.   Moreover  the linear  map $ T: Y_\sigma \rightarrow \tilde{X}_\sigma$, given by $ T(f)=\phi$,  is continuous.

\end{thm*}

For our approach we won't work directly with $\tilde{X}_\sigma$ but instead work with a slight variant that allows us to handle the advection term.
  So towards this define the norm
 \begin{eqnarray*}
 \| \phi\|_{X_\sigma}&:=& \sup_{|x| \le 1} \left( |x|^\sigma | \phi(x)| + |x|^{\sigma +1} | \nabla \phi(x)| \right)   \\
 && + \sup_{|x|\ge 1} \left( |x|^\frac{2}{p-1} | \phi(x)| + |x|^{ \frac{2}{p-1}+1} | \nabla \phi(x)| \right)
 \end{eqnarray*} and let $ X_\sigma$ denote the completion of $ C_c^\infty(\IR^N \backslash \{0\})$ with respect to this norm.

 \begin{lemma} \label{homo} Suppose $ N \ge 4$ and   $ p> \frac{N+1}{N-3}$.     For sufficiently small $ \sigma>0$ and for all $ f \in Y_\sigma$ there exists some $ \phi \in X_\sigma$ such that $L(\phi)=f$  in $ \IR^N$. Moreover the linear  map $ T:Y_\sigma \rightarrow X_\sigma$ defined by $ T(f)=\phi$ is continuous.

 \end{lemma}

\noindent
\textbf{Proof of Lemma \ref{homo}.} Suppose $ f \in Y_\sigma$ and let $ \phi \in \tilde{X}_\sigma$ be such that $L(\phi)=f$ in $ \IR^N$.   Then there exists some $ C>0$, independent of $f$ and $ \phi$, such that $ \| \phi\|_{\tilde{X}_\sigma} \le C \| f\|_{Y_\sigma}$.    Our goal is to now  show there is some $ C_1>0$, independent of $ f$ and $ \phi$, such that $ \| \phi\|_{X_\sigma} \le C_1 \| f\|_{Y_\sigma}$ and this will complete the proof.     Define the re-scaled functions $ \phi_m(x)= \phi(x_m + r_m x)$ where $ |x_m|>0$, $ r_m= \frac{|x_m|}{4}$ for $ |x| <1$.  Note that
\[ -\Delta \phi_m(x)= p r_m^2 w(x_m+r_m x)^{p-1} \phi(x_m+r_m x) + r_m^2 f(_m+ r_m x)=:g_m(x), \] for all $ x \in B_1$.    We now obtain some estimates on $\phi_m$ using the following result, which is just an elliptic regularity result coupled with the Sobolev imbedding theorem:   for $ t>N$ there is some $ C_t$ such that

 \begin{equation} \label{ell_reg}
 \sup_{B_\frac{1}{2}} | \nabla \phi_m(x)| \le C_t \left( \int_{|x|<1} |\Delta \phi_m(x)|^t dx \right)^\frac{1}{t} + C_t \int_{|x|< 1} | \phi_m(x)| dx.
    \end{equation}
    We now assume we are in the case of $ |x_m| \ge 1$.  Using the fact that $ f \in Y_\sigma$ and $ \phi \in \tilde{X}_\sigma$ one sees that $ |x_m|^\frac{2}{p-1} |g_m(x)| \le C$ for all $ |x| <1$  and $m$.  Putting these estimates into (\ref{ell_reg}) gives $ \sup_{B_\frac{1}{2}} | \nabla \phi_m(x)| \le C |x_m|^\frac{-2}{p-1}$ and from this we see that
    \[ |x_m|^{\frac{2}{p-1}+1} | \nabla \phi(x_m)| \le C_1.\]   The case of $ |x_m| \le 1$ is handled as above.  Combining these results gives us the desired bound.

\hfill $ \Box$

\textbf{Proof of Theorem \ref{existence}, 1).}   To solve (\ref{eq}) we first consider solving the related problem given by
\begin{equation} \label{absol}
-\Delta u + a(x) \cdot \nabla u = |u|^p \qquad \mbox{ in } \IR^N.
\end{equation}  To do this we perturb off the radial solution $w$ of the advection free problem.  So we look for a solution of the form $ u=\phi+w$.  So we need to find a solution $ \phi$ of
\begin{equation} \label{absol_1}
L(\phi) = - a \cdot \nabla w - a \cdot \nabla \phi + |w+\phi|^p - p w^{p-1} \phi - w^p \qquad \mbox{ in $ \IR^N$},
\end{equation} where $ L(\phi)=-\Delta \phi - p w^{p-1} \phi$.  Letting $T$ be defined as in Lemma \ref{homo} we are looking for a $ \phi \in X_\sigma$ such that
 \begin{equation} \label{ma}
 \phi= -T(a \cdot \nabla w) - T(a \cdot \nabla \phi) + T( |w+\phi|^p - p w^{p-1} \phi - w^p).
   \end{equation}  To find such a $\phi$ we define $ J(\phi)$ to be the mapping on the right hand side of (\ref{ma}) and we will now show that for a suitable $R$ that $J$ is a contraction mapping on the closed ball $B_R$, centered at the origin, in $X_\sigma$.  We will then argue that $u=w+\phi$ is positive.   We begin by showing $J$ is into $B_R$.   In what follows $C$ can depend on $p,a,w $ but not on $x,R,\phi$ and $ \sigma$ provided $ \sigma$ is small. Let $ R>0$ and let $ \phi \in B_R$. Then note that that there is some $ C>0$ such that
    \begin{equation} \label{last}
    \|J(\phi)\|_{X_\sigma} \le C \|a \cdot \nabla w\|_{Y_\sigma} + C \| a \cdot \nabla \phi\|_{Y_\sigma} + C\| |w+\phi|^p- pw^{p-1} \phi - w^p\|_{Y_\sigma}.
    \end{equation}
   We now estimate the terms on the right hand side.
   \begin{eqnarray*}
   \| a \cdot \nabla w\|_{Y_\sigma} & \le & \sup_{|x| \le 1} |a(x)| |x| \sup_{|x| \le 1} |x|^{1+\sigma} | \nabla w(x)| \\ &&  + \sup_{|x|  \ge 1} |x||a(x)| \sup_{|x| \ge 1} |x|^{ \frac{2}{p-1}+1} | \nabla w(x)| \\
   & \le & ( \sup_{x} |a(x)||x|) \| w\|_{X_\sigma} \\
   %& \le &(2 \sup_{x} |a(x)||x|) C.
   \end{eqnarray*}  The same argument shows that
   \[ \| a \cdot \nabla \phi\|_{Y_\sigma} \le ( \sup_{x} |x| |a(x)|) \| \phi \|_{X_\sigma}.\] We now approximate the last term in   (\ref{last}).
    For this we need the following real analysis result.  There exists some $ C=C_p$ such that for all numbers $ w>0$ and $ \phi \in \IR$ we have
   \[ \big| |w+\phi|^p- p w^{p-1} \phi - w^p \big| \le C \left( w^{p-2} \phi^2 + | \phi|^p \right).\]  Set   $ \Gamma= |w+\phi|^p-p w^{p-1} \phi - w^p$. Then one sees
   \begin{eqnarray*}
    \| \Gamma \|_{Y_\sigma} & \le & C \sup_{|x| \le 1} |x|^{\sigma +2} \left( w^{p-2} \phi^2 + | \phi|^p \right) \\
    && + C \sup_{|x| \ge 1} |x|^{\frac{2}{p-1}+2} \left( w^{p-2} \phi^2 + | \phi|^p \right) \\
    & := & C I_1 + C I_2.
    \end{eqnarray*}  Then note that
    \begin{eqnarray*}
    I_1 & = & \sup_{|x| \le 1} \left( |x|^{2 - \sigma} w^{p-2} \left( |x|^\sigma \phi(x) \right)^2 + |x|^{ \sigma+2-\sigma p} \left( | \phi(x)| |x|^\sigma \right)^p \right)  \\
    & \le & \sup_{|x| \le 1} \left(|x|^{2-\sigma} w^{p-2} \| \phi\|_{X_\sigma}^2 + |x|^{\sigma +2 - \sigma p} \| \phi\|_{X_\sigma}^p \right) \\
    & \le & C \| \phi\|_{X_\sigma}^2  + C \| \phi\|_{X_\sigma}^p
    \end{eqnarray*} for sufficiently small $ \sigma >0$. One can similarly show that
     \begin{eqnarray*}
     I_2 &\le& \sup_{|x| \ge 1} \left( |x|^\frac{2}{p-1} w \right)^{p-2} \| \phi\|_{X_\sigma}^2 + \| \phi\|_{X_\sigma}^p  \\
     & \le & C \| \phi\|_{X_\sigma}^2  + \| \phi\|_{X_\sigma}^p.
     \end{eqnarray*}  So combining these results we arrive at
     \begin{eqnarray} \label{poo}
      \|J(\phi)\|_{X_\sigma} & \le &  C \sup_x |x| |a(x)|  + C \sup_x |x| |a(x)| \| \phi\|_{X_\sigma} \nonumber \\
      && + C \| \phi\|_{X_\sigma}^2 + C \| \phi\|_{X_\sigma}^p.
      \end{eqnarray}

       Before choosing $R$ we examine the condition on $J$ to be a contraction on $B_R$. First note there is some $ C=C_p$ such that for all numbers $ w >0$ and $ \hat{\phi}, \phi \in \IR$ one has
      \begin{equation} \label{contr}
      \Big|  | \hat{\phi} +w|^p - | \phi + w|^p - p w^{p-1}(\hat{\phi}-\phi) \big| \le  C M | \hat{\phi}-\phi|
      \end{equation} where
      \[ M=   w^{p-2} \left( | \hat{\phi}| + | \phi| \right) + | \hat{\phi}|^{p-1} | + | \phi|^{p-1} .\] 
        Let  $ \hat{\phi}, \phi \in B_R$.  Then
      \[ J(\hat{\phi})- J(\phi)= -T( a \cdot \nabla ( \hat{\phi}- \phi)) + T( |w+\hat{\phi}|^p- |w+\phi|^p - p w^{p-2} (\hat{\phi}- \phi)),\]  and so

      \begin{eqnarray*}
      \| J(\hat{\phi})- J(\phi) \|_{X_\sigma} & \le & C \| a \cdot \nabla ( \hat{\phi}- \phi) \|_{Y_\sigma}  \\
      &&+ C \| |w+\hat{\phi}|^p- |w+\phi|^p- p w^{p-1}( \hat{\phi}-\phi)\|_{Y_\sigma} \\
      & =:&  C J_1 + C J_2
      \end{eqnarray*} Arguing as above one easily sees that $ J_1 \le \sup_x( |x||a(x)|) \| \hat{\phi}- \phi\|_{X_\sigma}$.   Using (\ref{contr}) we see that
      \[ J_2 \le C  \sup_{|x| \le 1} |x|^{2+\sigma} M | \hat{\phi}-\phi|+ C \sup_{|x| \ge 1} |x|^{\frac{2}{p-1}+2} M | \hat{\phi}-\phi| =: CJ_3+ C J_4.\]   We now compute the various terms in $J_3$ and $J_4$.
      \begin{eqnarray*}
      \sup_{|x| \le 1} |x|^{2+\sigma} w^{p-1} | \hat{\phi}| | \hat{\phi} - \phi| & \le & \sup_{|x| \le 1} ( |x|^{2-\sigma} w^{p-2} ) \| \hat{\phi}\|_{X_\sigma} \| \hat{\phi}- \phi\|_{X_\sigma}.
      \end{eqnarray*}  Also we have
       \begin{eqnarray*}
      \sup_{|x| \le 1} |x|^{2+\sigma} | \hat{\phi}|^{p-1} | \hat{\phi} - \phi| & \le & \sup_{|x| \le 1} |x|^{2-\sigma-\sigma(p-1)} \| \hat{\phi}\|_{X_\sigma}^{p-1} \| \hat{\phi}- \phi\|_{X_\sigma}, \\
      & \le & \| \hat{\phi}\|_{X_\sigma}^{p-1} \| \hat{\phi}- \phi\|_{X_\sigma},
      \end{eqnarray*} for small enough $ \sigma>0$.   Combining these results we obtain
      \begin{eqnarray*}
      J_3 &\le& \left( \sup_{|x| \le 1} |x|^{2-\sigma} w^{p-2} 2R + 2 R^{p-1} \right) \| \hat{\phi}- \phi\|_{X_\sigma} \\
      & \le & \left( C R + 2 R^{p-1} \right)  \| \hat{\phi}- \phi\|_{X_\sigma}.
      \end{eqnarray*}  One can argue in a similar fashion to show
      \begin{eqnarray*}
      J_4 & \le & \sup_{|x| \ge 1} \left( |x|^\frac{2}{p-1} w \right)^{p-2} ( \| \hat{\phi}\|_{X_\sigma} + \| \phi\|_{X_\sigma} ) \| \hat{\phi}-\phi\|_{X_\sigma} \\
      && + \left( \| \hat{\phi} \|_{X_\sigma}^{p-1} + \| {\phi} \|_{X_\sigma}^{p-1} \right)\| \hat{\phi}-\phi\|_{X_\sigma} \\
      & \le & \left(  C R  + 2 R^{p-1} \right) \| \hat{\phi}-\phi\|_{X_\sigma}.
      \end{eqnarray*}  Combining the results we obtain an inequality of the form
      \begin{equation} \label{contraction99}
      \| J(\hat{\phi})-J(\phi)\|_{X_\sigma} \le C \left( \sup_x |x| |a(x)| + R + R^{p-1} \right) \| \hat{\phi}- \phi\|_{X_\sigma}.
      \end{equation}  We now pick $R$ and put conditions on $a$.  Fix $ R$ sufficiently small such that $ CR^2 + CR^p \le \frac{R}{10}$ and such that $ CR + CR^{p-1} < \frac{1}{2}$. Now impose a smallness condition on $ a$ such that $  C \sup_x |x| |a(x)| +  C \sup_x |x| |a(x)|R \le \frac{R}{10}$ and $ C \sup_x |x| |a(x)| < \frac{1}{4}$.  These conditions are sufficient to show that $J$ is a contraction mapping on $B_R$ in $ X_\sigma$ and hence by the Contraction Mapping Principle there is some $ \phi \in B_R$ such that $J(\phi)=\phi$, which was the desired result.  So we have $\phi \in B_R$ such that
 \begin{equation} \label{almost}
 -\Delta ( w + \phi ) + a \cdot \nabla ( w+ \phi) = |w+ \phi|^p \qquad \mbox{in } \IR^N.
  \end{equation}
  By taking $ R>0$ smaller, which imposes a further smallness condition on $a$, we can assume that
 \begin{equation} \label{close}
 \sup_{|x| \ge 1} |x|^\frac{2}{p-1} | \phi(x)| \le \frac{1}{10} \inf_{|y| \ge 1} |y|^\frac{2}{p-1}w(y).
  \end{equation} Using this one sees that $ \phi + w >0$ on $ |x| \ge 1$.  Note there are some possible regularity issues for $ \phi$ near the origin.  But taking $ \sigma >0$ small enough and applying elliptic regularity theory, along with a bootstrap, one sees that $ \phi$ is  at least $ C^{2,\alpha}$ in a ball around the origin for some $ \alpha >0$.  One can now apply the maximum principle to see that $ u=w+\phi$ is a positive solution of (\ref{eq}).

\hfill  $\Box$

\textbf{Proof of Theorem \ref{existence}, 2).} First note that a computation shows that $ p_c >\frac{N+1}{N-3}$. For $ R>0$  sufficiently small there is some $u_R>0$ which satisfies (\ref{eq}) and as $R$ gets small one imposes smallness conditions on $a$.   For $ m \ge 2$ an integer let $ E=E_{m,R}>0$ denote the first eigenfunction of $L(E):=-\Delta E + a \cdot \nabla E - p u_R^{p-1} E $ on the ball $ B_m$ with $E=0$ on $ \partial B_m$ and let $ \mu_{m,R}$ denote the first eigenvalue.
  We now multiply the equation for $E$ by $E$ and integrate over $B_m$. Using the fact that $a$ is divergence free (this is only spot we utilize this fact) one sees, after a suitable $L^2$ normalization of $E$, that
\[ \int | \nabla E|^2 = \int p u_R^{p-1} E^2 + \mu_{m,R}.\]  We now extend $E$ outside $B_m$ by setting it to be zero and we use the fact that $w$ satisfies (\ref{superstable}) to arrive at   
\[ (p+\E) \int w^{p-1} E^2 \le  \int p u_R^{p-1} E^2 + \mu_{m,R},\] for some fixed $ \E>0$.
   Note that $ u_R \rightarrow w$ in $ X_\sigma$ as $ R \rightarrow 0$ and so we can argue
   as in (\ref{close}), that for any $ \delta >0$ we can pick $R$ small enough such that $ u_R(x) \le (\delta +1) w(x)$ for all $ |x| \ge 1$.
     Using elliptic regularity and Sobolev imbedding one sees that the restriction of $u_R$ to the unit ball converges to the restriction of $w$  uniformly. And so we can assume that $ u_R^{p-1} \le w^{p-1} +\delta$ for $ |x| \le 1$  for small enough $R$.  Using this estimates and breaking the integrals into the regions $ |x| \ge 1$ and $ |x| \le 1$ one arrives at
   \[
     \left( (p+\E) -p(1+\delta)^{p-1} \right) \int_{|x| \ge 1} w^{p-1} E^2 +(\E- p \delta) \int_{|x| \le 1} w^{p-1} E^2 \le \mu_{m,R},\] for sufficiently small $ R$. Now by taking $ \delta >0$ small enough one sees that for fixed $ R$ small enough we have $ \mu_{m,R} \ge 0$ for all $ m \ge 2$.   We now fix this $R$ and let $ u=u_R$,  $ E_m=E_{m,R} $ and $ \mu_m= \mu_{m,R}$.   So we have that $ E_m >0$ satisfies
      \begin{eqnarray*}
 \left\{ \begin{array}{lcl}
\hfill   -\Delta E_m +a(x) \cdot \nabla E_m   &=& p u^{p-1} E_m + \mu_m E_m \qquad \mbox{ in } B_m  \\
\hfill E_m &=& 0 \qquad  \qquad \qquad \quad \quad \mbox{ on } \partial B_m.
\end{array}\right.
  \end{eqnarray*} Lets assume   that $ \mu_m \rightarrow 0$.   By suitably scaling $E_m$ we can assume that $ E_m(0)=1$.  Now fix $ k \ge 1$ an integer and let $ m \ge k+2$.  Now note that $E_m$ satisfies the same equation on $ B_{k+1}$ and hence by Harnacks inequality there is some $C_k>0$ such that
  \[ \sup_{B_k} E_m \le C_k \inf_{B_k} E_m \le C_k,\] for all $ m \ge k+2$.   Using elliptic regularity one can show that $ E_m$ is bounded in $ C^{1,\alpha}(B_k)$ and by a diagonal argument there is some subsequence of $E_m$, which we still denote by $ E_m$,  which converges to some $ E \ge 0$ locally in $C^{1,\beta}$ for some $ \beta>0$ and  $E(0)=1$.  One can then argue that $ E$ satisfies
  \[ -\Delta E + a(x) \cdot \nabla E = p u^{p-1} E \qquad \mbox{in $ \IR^N$,} \] and then we can apply the strong maximum principle to see that $ E>0$.    This shows that $u$ is a stable solution of (\ref{eq}) which was the desired result.    We now show $ \mu_m \rightarrow 0$.  We begin by putting $ E_m$, which we $L^2$ normalize, into (\ref{hardy_mine}) with $ \beta = \frac{1}{2}$ to arrive at
  \[ \mu_m \int \phi^2 + \frac{1}{2} \int \frac{|\nabla E_m|^2}{E_m^2} \phi^2 \le 2 \int | \nabla \phi |^2 + \int \frac{ a \cdot \nabla E_m}{E_m} \phi^2,\] for all $ \phi \in C_c^\infty(B_m)$.  We now use Young's inequality to arrive at
  \begin{eqnarray*}
  \mu_m \int \phi^2 + \frac{1}{2} \int \frac{| \nabla E_m|^2}{E_m^2} \phi^2 & \le & 2 \int | \nabla \phi^2| \\
  && +\E \int \frac{| \nabla E_m|^2}{E_m^2} \phi^2  \\
  && + \frac{1}{4\E} \int | a|^2 \phi^2.
  \end{eqnarray*}   By taking $ \E>0$ small enough and re-grouping terms  and by using the fact that $ |a(x)| \le \frac{C^2}{|x|^2}$ along with Hardy's inequality, one can obtain
  \[ \mu_m \int \phi^2 \le C \int | \nabla \phi|^2 \qquad \forall \phi \in C_c^\infty(B_m),\] where $C$ is independent of $m$. From this we can conclude that $ \limsup_m \mu_m \le 0$ but we already have $ \mu_m \ge 0$ and hence we have the desired result.

\hfill  $\Box$


\begin{thebibliography}{99}





\bibitem{cabreent}  X. Cabr\'e and A. Capella, \emph{On the stability of radial solutions of semilinear elliptic equations in all of $R^n$}, . C. R. Acad. Sci. Paris, Ser. I 338 (2004), 769-774.



 \bibitem{Caf} L. Caffarelli, B. Gidas and J. Spruck. \emph{Asymptotic symmetry and local behaviour of semilinear elliptic equations with critical Sobolev growth}. Commun. Pure Appl. Math. 42 (1989), 271–297.




\bibitem{ces}  D. Castorina, P. Esposito and B.  Sciunzi; \emph{Low dimensional instability for semilinear and quasilinear problems in $\Bbb R^N$}.  Commun. Pure Appl. Anal.  8  (2009),  no. 6, 1779-1793.




\bibitem{chen} W. Chen and C. Li, \emph{Classification of solutions of some nonlinear elliptic equations}. Duke Math. J. 63 (1991), 615–622.




\bibitem{craig} C. Cowan,
\emph{Optimal Hardy inequalities for general elliptic operators with improvements}, Commun. Pure Appl. Anal. 9 (2010), no. 1, 109-140.


\bibitem{Cowan_fazly} C. Cowan and M. Fazly, \emph{On stable entire solutions of semilinear elliptic equations with weights}, Proc. Amer. Math. Soc. 140 (2012), 2003-2012

\bibitem{advection} C. Cowan and N. Ghoussoub, \emph{Regularity of the extremal solution in a MEMS model with advection}. Methods Appl. Anal. (2008) 8pp.




\bibitem{davila} J. D\'avila, M. del Pino and M. Musso, \emph{The Supercritical Lane–Emden–Fowler Equation in Exterior Domains}, Communications in Partial Differential Equations, 32:8, 1225-1243, (2007).






\bibitem{davila_fast} J. D\'avila, M. del Pino, M. Musso and J. Wei, \emph{Fast and slow decay solutions for supercritical elliptic problems in exterior domains},
Calculus of Variations and Partial Differential Equations
August 2008, Volume 32, Issue 4, pp 453-480.










\bibitem{e1} P. Esposito, \emph{Linear instability of entire solutions for a class of non-autonomous elliptic equations.}  Proc. Roy. Soc. Edinburgh Sect. A  138  (2008),  no. 5, 1005-1018.



\bibitem{e2} P. Esposito, \emph{Compactness of a nonlinear eigenvalue problem with a singular nonlinearity.}  Commun. Contemp. Math.  10  (2008),  no. 1, 17-45.





\bibitem{egg} P. Esposito, N. Ghoussoub and Y. Guo,
 \emph{Mathematical analysis of partial differential equations modeling electrostatic MEMS}, Courant Lecture Notes in Mathematics, 20. Courant Institute of Mathematical Sciences, New York; American Mathematical Society, Providence, RI, 2010. xiv+318 pp.



 \bibitem{zz}
 P. Esposito, N. Ghoussoub and  Y. Guo,
 \emph{ Compactness along the branch of semistable and unstable solutions for an elliptic problem with a singular nonlinearity}. Comm. Pure Appl. Math. 60 (2007), no. 12, 1731-1768.




\bibitem{farina} A. Farina,
 \emph{On the classification of solutions of the Lane-Emden equation on unbounded domains of $\Bbb R^N$}, J. Math. Pures Appl. (9) 87 (2007), no. 5, 537-561.



\bibitem{Gidas}  B. Gidas, W. Ni and L. Nirenberg, \emph{Symmetry and related properties via the maximum principle}. Commun. Math. Phys. 68 (1979), 525–598. MR0544879 (80h:35043)




\bibitem{gidas}B. Gidas and J. Spruck, \emph{Global and local behavior of positive solutions of nonlinear elliptic equations.} Comm. Pure Appl. Math., 34(4):525-598, 1981.



\bibitem{Gui_Ni_Wang} C. Gui, W. M.  Ni and X.  Wang,  \emph{On the stability and instability of positive steady states
of a semilinear heat equation in $\IR^n$.}  Comm. Pure Appl. Math. 45:1153-1181.






\bibitem{Joseph_lundgren} D.D. Joseph and T.S. Lundgren, \emph{Quasilinear Diricblet Problems Driven
by Positive Sources},  Archive for Rational Mechanics and Analysis,
23.II.1973, Volume 49, Issue 4, pp 241-269.



\bibitem{harnack} I. Kukavica, M. Ignatova and L. Ryzhik, \emph{The Harnack inequality for second-order elliptic equations with divergence-free drifts}, Preprint, 2012.





\bibitem{advect_2} X. Luo, D. Ye  and F. Zhou, \emph{Regularity of the extremal solution for some elliptic problems with singular nonlinearity and advection},
Journal of Differential Equations, Volume 251, issue 8 (October 15, 2011), p. 2082-2099.






\bibitem{Wang_solo} X. Wang, \emph{On the Cauchy Problem for Reaction-Diffusion Equations},
Transactions of the American Mathematical Society, Vol. 337, No. 2 (June 1993) pp. 549-590.








\end{thebibliography}
\end{document}